\documentclass[12pt]{amsart}
\usepackage{amssymb,amsmath, amsthm,latexsym}
\usepackage{a4wide}
\usepackage{verbatim}

\theoremstyle{plain}

\newtheorem{lem}{Lemma}[section]
\newtheorem{theo}[lem]{Theorem}
\newtheorem{prop}[lem]{Proposition}

\parskip=\medskipamount

  \newcommand {\C}{{\mathbb C}}
  \newcommand {\bH}{{\mathbb H}}
  \newcommand {\N}{{\mathbb N}}
  \newcommand {\R}{{\mathbb R}}

  \newcommand {\af}{{\mathfrak a}}
  \newcommand {\gf}{{\mathfrak g}}
  
  \newcommand {\kf}{{\mathfrak k}}
  \newcommand {\tf}{{\mathfrak t}}
  
  \newcommand {\hf}{{\mathfrak h}}
  
  \newcommand {\nf}{{\mathfrak n}}
  
  \newcommand {\pg}{{\mathfrak p}}

\renewcommand {\H}{{\mathcal H}}

 \newcommand {\cS}{{\mathcal S}}

\newcommand {\bs}{\backslash}

\renewcommand{\Im}{\operatorname{Im}}
\renewcommand{\Re}{\operatorname{Re}}
\newcommand{\Tr}{\operatorname{Tr}}
\newcommand{\ad}{\operatorname{ad}}
\newcommand{\End}{\operatorname{End}}

\newcommand{\rank}{\operatorname{rank}}
\newcommand{\tr}{\operatorname{tr}}
\newcommand{\Id}{\operatorname{Id}}
\newcommand{\Hom}{\operatorname{Hom}}

\newcommand{\I}{\operatorname{I}}
\newcommand{\rk}{\operatorname{rank}}
\newcommand{\rkk}{\operatorname{rk}}
\newcommand{\vol}{\operatorname{vol}}

\newcommand{\SL}{\operatorname{SL}}
\newcommand{\GL}{\operatorname{GL}}

\newcommand{\Ad}{\operatorname{Ad}}

\newcommand{\supp}{\operatorname{supp}}

\newcommand{\spec}{\operatorname{spec}}
\newcommand{\rP}{P^{1/2}}
\newcommand{\rD}{\Delta^{1/2}_E}
\newcommand{\rep}{\eta}

\begin{document}
\title[]
{The Selberg trace formula for non-unitary representations of the lattice}
\date{\today}

\author{Werner M\"uller}
\address{Universit\"at Bonn\\
Mathematisches Institut\\
Beringstrasse 1\\
D -- 53115 Bonn, Germany}
\email{mueller@math.uni-bonn.de}

\begin{abstract}
Let $X=\Gamma\bs G/K$ be a compact locally symmetric space. In this paper we
establish  a version of the Selberg trace formula for non-unitary 
representations of the lattice $\Gamma$. On the spectral side appears the 
spectrum of the ``flat Laplacian'' $\Delta^\#$, acting in the space of 
sections of the associated flat bundle. In general, this is a non-self-adjoint
operator. 
\end{abstract}

\maketitle

\section{Introduction}
\setcounter{equation}{0}
Let $G$ be a connected real semisimple Lie group with finite center and of 
non-compact type. Let $K$ be a maximal compact subgroup of $G$. Then $S=G/K$
is a Riemannian symmetric space of nonpositive curvature. We fix an invariant 
metric on $S$ which we normalize using the Killing form. Let $\Gamma\subset G$
be a discrete subgroup such that $\Gamma\bs G$ is compact. For simplicity we
assume that $\Gamma$ is torsion free. Then $\Gamma$ acts properly 
discontinuously on $S$ and $X=\Gamma\bs S$ is a compact locally symmetric
manifold. 

Let $\chi\colon\Gamma\to\GL(V_\chi)$ be a finite-dimensional unitary
representation. Denote by $E_\chi\to\Gamma\bs S$ the associated flat vector
bundle. It is equipped with a canonical Hermitian fiber metric $h^\chi$ and
a compatible flat connection $\nabla^\chi$. Let $d_\chi\colon C^\infty(X,E_\chi)
\to \Lambda^1(X,E_\chi)$ be the associated exterior derivative and let 
$\delta_\chi$ be the formal adjoint of $d_\chi$ with respect to the inner
products in $C^\infty(X,E_\chi)$ and 
$C^\infty(X,T^*(X)\otimes E_\chi)$, respectively,  induced
by the invariant metric on $S$ and the fiber metric $h^\chi$ in $E_\chi$.
Let $\Delta_\chi=\delta_\chi d_\chi$ be the associated Laplace operator. 
Then $\Delta_\chi$ is a second order elliptic, formally self-adjoint,
 nonnegative
differential operator. The Selberg trace formula computes the distributional
trace of $\cos t\sqrt{\Delta_\chi}$ in terms of a sum of distributions on $G$, 
which are associated to  the conjugacy classes of $\Gamma$. 

The trace formula has many applications. Of  particular interest for the 
present paper are applications to Ruelle and Selberg zeta functions.
Especially the analytic continuation and the
functional equation of twisted Ruelle and Selberg zeta functions rely on the
twisted Selberg trace formula \cite{BO}, \cite{Se2}. 
Also spectral invariants of 
locally symmetric spaces such as analytic torsion and eta invariants can
be studied with the help of the trace formula (see \cite{Fr}, \cite{Mil}, 
\cite{MS1}, \cite{MS2}). So far, these applications
are restricted to unitary representations of $\Gamma$ and it this very
desirable to extend the scope of the trace formula so that all
finite-dimensional representations are covered.  
This is the main goal of this paper.


To begin with we recall the trace formula for a unitary representation $\chi$
(see \cite{Se1}, \cite{Se2}).
Let $\spec(\Delta_\chi)$ be the spectrum of $\Delta_\chi$. It consists of a
sequence $0\le\lambda_1<\lambda_2<\cdots$ of eigenvalues of finite 
multiplicities. Denote by $m(\lambda_k)$ the multiplicity of $\lambda_k$.
Let $\varphi\in\cS(\R)$ be even and assume that the Fourier transform 
$\hat\varphi$ of $\varphi$ belongs to $C^\infty_c(\R)$. 
Then $\varphi((\Delta_\chi)^{1/2})$ is a trace class operator and its trace is 
given by
\begin{equation}\label{traceequ}
\Tr\varphi((\Delta_\chi)^{1/2})=\sum_{\lambda\in\spec(\Delta_\chi)}
m(\lambda)\varphi(\lambda^{1/2}).
\end{equation}
Let $\tilde\Delta$ be the Laplacian of $S$, and let $h_\varphi$ be the kernel 
of the invariant integral operator $\varphi(\tilde\Delta^{1/2})$. It belongs to 
the space $C^\infty_c(G/\hskip-2pt/K)$ of $K$-bi-invariant compactly supported 
smooth functions on $G$. 
Given $\gamma\in\Gamma$ let $\{\gamma\}_\Gamma$ denote its $\Gamma$-conjugacy
class. Furthermore, let $G_\gamma$ and $\Gamma_\gamma$ denote the centralizer of
$\gamma$ in $G$ and $\Gamma$, respectively. Then the first version of the 
trace formula is the following identity.
\begin{equation}\label{trace0}
\begin{split}
\sum_{\lambda\in\spec(\Delta_\chi)}m(\lambda)\varphi(\lambda^{1/2})
&=\vol(\Gamma\bs S)\dim V_\chi h_\varphi(e)\\
&+\sum_{\{\gamma\}_\Gamma\not=e}\tr\chi(\gamma)\vol(\Gamma_\gamma\bs G_\gamma)
\int_{G_\gamma\bs G}h_\varphi(g^{-1}\gamma g)\;d\dot g.
\end{split}
\end{equation}
To make this formula more explicite, one can use the Plancherel formula to
express $h_\varphi$ in terms of $\varphi$. Furthermore, the orbital integrals
\[
I(g;\varphi)=\int_{G_\gamma\bs G}h_\varphi(g^{-1}\gamma g)\;d\dot g
\]
are invariant distributions and therefore, one can use Harish-Chandra's Fourier
inversion formula to compute them (see \cite[\S 4]{DKV}). In the
higher rank case this is rather complicated and no closed formula is available.
In the rank one case, however, the situation is much better. There is a
simple formula expressing the orbital integrals in terms of characters which
leads to an explicit form of the trace formula \cite[Theorem 6.7]{Wa}.

To extend the Selberg trace formula to all
finite-dimensional representations of $\Gamma$, we first note that the sum on 
the right hand side of (\ref{trace0}) is 
finite and therefore, it is well defined for all finite-dimensional 
representations $\chi$. The question is what is the appropriate operator which
replaces the Laplacian on the
left hand side. In general there is no Hermitian metric on $E_\chi$ 
which is compatible with the flat connection $\nabla^\chi$. 
A special case has been studied by Fay 
\cite{Fa}. He considered the analytic torsion $T_M(\chi)$ of a Riemann
surface $M=\Gamma\bs \bH$ of genus $g>1$ and a unitary character $\chi\in
\Hom(\Gamma,S^1)$ and established the analytic continuation of $T_M(\chi)$ 
to all characters $\chi\in\Hom(\Gamma,\C^*)$. To this end he introduced a
non-self-adjoint Laplacian. We use a similar approach in the general case. 
The operator that replaces $\Delta_\chi$
is the ``flat Laplacian'' $\Delta_\chi^\#$ which is defined as follows. Let
$\ast\colon \Lambda^p(T^\ast X)\to \Lambda^{n-p}(T^\ast X)$
be the Hodge star operator associated to the Riemannian metric of $X$. 
Extend $\ast$ to an operator  $\ast_\chi$ in $\Lambda^p(T^\ast X)\otimes E_\chi$
 by $\ast_\chi=\ast\otimes \Id_{E_\chi}$. Define
$\delta^\#_\chi:=(-1)^{n+1}\ast_\chi d_\chi\ast_\chi.$
Then the flat Laplacian $\Delta^\#$ is defined as
\[\Delta^\#_\chi=\delta^\#_\chi d_\chi.\]
If $\chi$ is unitary, $\Delta_\chi^\#$ equals $\Delta_\chi$. For an arbitrary
$\chi$ 
we pick any Hermitian fiber metric in $E_\chi$ and use it together with the
Riemannian metric on $X$ to introduce an inner product in 
$C^\infty(X, E_\chi)$. 
In general, $\Delta_\chi^\#$ is a not self-adjoint w.r.t. this inner product.
However, if we define the corresponding 
Laplace operator $\Delta_{\chi}$ as above by $\delta_\chi d_\chi$, where 
the formal adjoint $\delta_\chi$ is taken w.r.t. to the inner product, 
then $\Delta_\chi^\#$ has the
same principal symbol as $\Delta_{\chi}$. This implies that the operator
$\Delta_\chi^\#$ has nice spectral properties. Its spectrum  is discrete and 
contained in a positive cone $C\subset \C$ with $\R^+\subset C$ 
(see \cite{Sh}). 
It follows  that an Agmon angle $\theta$ exists for 
$\Delta_\chi^\#$ and we can define 
$\varphi\left((\Delta_\chi^\#\right)^{1/2}_\theta)$ by the usual functional 
calculus \cite{Sh}. This is a trace class operator. 
Since $\varphi$ is assumed to be even, 
$\varphi\left((\Delta_\chi^\#\right)^{1/2}_\theta)$ is independent of $\theta$
and we can delete $\theta$ from the notation. Lidkii's theorem 
\cite[Theorem 8.4]{GK} generalizes (\ref{traceequ}). As mentioned above, the 
spectrum $\spec(\Delta_\chi^\#)$ of $\Delta_\chi^\#$ is 
discrete and consists of eigenvalues only. For 
$\lambda\in\spec(\Delta_\chi^\#)$ let $m(\lambda)$ denote
the algebraic multiplicity of $\lambda$, i.e., $m(\lambda)$ is the  dimension 
of the root space which consists of all $f\in C^\infty(X,E_\chi)$ such that there
is $N\in\N$ with $(\Delta^\#_\chi-\lambda\I)^Nf=0$.  
Then by Lidskii's theorem we have
\begin{equation}\label{traceequ1}
\Tr\varphi((\Delta_\chi^\#)^{1/2})=\sum_{\lambda\in\spec(\Delta_\chi^\#)}
m(\lambda)\varphi(\lambda^{1/2}).
\end{equation}
The first version of our trace formula 
generalizes (\ref{trace0}) with 
$\Tr \varphi\left((\Delta_\chi^\#\right)^{1/2})$
on the left hand side. 

Actually, we prove a more general result. Let $\tau$ be an irreducible 
representation of $K$ and $E_\tau\to\Gamma\bs S$ the associated
locally homogeneous vector bundle, equipped with its canonical
invariant connection $\nabla^\tau$. Let $\nabla=\nabla^{\tau,\chi}$ be the 
product connection in $E_\tau\otimes E_\chi$, and let $\Delta_{\tau,\chi}^\#=
-\Tr(\nabla^2)$ be the corresponding connection Laplacian. Then for $\varphi$
as above $\varphi\left((\Delta_{\tau,\chi}^\#)^{1/2}\right)$ is a trace class
operator and we establish a trace formula for this operator 
which is similar to the scalar case.

If $G$ has split rank one, we get an explicite version of the trace formula.
To describe it we need to introduce some notation.
Let $G=KAN$ be an Iwasawa decomposition of $G$. Then $\dim A=1$. Let $\af$ be
the Lie algebra of $A$. The restriction of the Killing form to $\af^*$ 
defines an inner product on $\af^*$. Let $|\rho|$ denote the norm  of the 
half-sum $\rho$ of positive roots of $(G,A)$. Let
 $\gamma\in\Gamma\setminus\{e\}$. Then there is a unique closed geodesic 
$\tau_\gamma$ that corresponds to the $\Gamma$-conjugacy class 
$\{\gamma\}_\Gamma$ of $\gamma$.
 Denote by $l(\gamma)$ the length of $\tau_\gamma$. 
Furthermore, let $\gamma_0\in \Gamma$ be the unique primitive element such that
$\gamma=\gamma_0^k$ for some $k\in\N$.  Finally let $D(\gamma)$ be the 
discriminant of $\gamma$ (see (\ref{discr}) for its definition). 
Let $\beta(\lambda) d\lambda$ be the Plancherel measure for
spherical functions on $G$ \cite{He}. We can now state our main result
in the scalar case  which is the following theorem.

\begin{theo}\label{scalar}
For all even $\varphi\in\cS(\R)$ with $\widehat\varphi\in C_c^\infty(\R)$ we
have
\begin{equation}\label{traceform5}
\begin{split}
\sum_{\lambda\in\spec\left(\Delta_\chi^\#\right)}m(\lambda)
 \varphi\left((\lambda-|\rho|^2)^{1/2}\right)=&
\dim(V_\chi)\frac{\vol(\Gamma\bs S)}{2}\int_{\R} 
\varphi(\lambda)\beta(\lambda)\;d\lambda\\
&+\sum_{\{\gamma\}_\Gamma\not=e}\tr\chi(\gamma)
\frac{l(\gamma_0)}{D(\gamma)}\widehat\varphi(l(\gamma)).
\end{split}
\end{equation}
\end{theo}
Note that for every $c>0$ there are only finitely many conjugacy classes
$\{\gamma\}_\Gamma$ with $l(\gamma)\le c$. Therefore the sum on the right hand 
side is finite.

To describe our method we restrict attention to the scalar case, i.e, we
consider the operator $\Delta_\chi^\#$. 
Our method  is based on the approach of 
Bunke and Olbrich \cite{BO} to the Selberg trace formula in the unitary case.
We consider the wave equation
\begin{equation}\label{wave1}
\left(\frac{\partial^2}{\partial t^2}+\Delta_\chi^\#\right)u(t)=0,\quad
u(0)=f,\; u_t(0)=0,
\end{equation}
for any initial conditions $f\in C^\infty(X,E_\chi)$. 
Since the principal symbols 
of $\Delta_{\chi^\#}$ is given by $\sigma(x,\xi)=\parallel\xi\parallel^2
\Id_{E_x}$, the
operator $L=\frac{\partial^2}{\partial t^2}+\Delta_\chi^\#$ is strictly
hyperbolic in the sense of \cite[Chapt. IV, \S3]{Ta1}. Therefore
(\ref{discr}) has a unique solution $u(t;f)$. Let $\varphi\in\cS(\R)$ be even 
such that 
$\hat\varphi\in C^\infty_c(\R)$. Then it follows that 
\begin{equation}\label{intoper}
\varphi\left((\Delta_\chi^\#)^{1/2}\right)f=\frac{1}{\sqrt{2\pi}}
\int_\R\hat\varphi(t)u(t;f)\ dt.
\end{equation}
Let $\tilde u(t;f)$ and $\tilde f$ denote the lift of $u(t,f)$ and $f$,
respectively, to the
universal covering $S$ of $X$. Then the corresponding wave equation on $S$
with initial conditions $u(0)=\tilde f$, $u_t(0)=0$ 
is also strictly hyperbolic and by finite propagation speed it follows that it
has a unique solution $u(t;\tilde f)$. Thus we obtain
 $\tilde u(t,\tilde x;f)=u(t,\tilde x;\tilde f)$.  Since the lift of $E_\chi$
to $S$ is trivial, the lifted operator $\tilde\Delta^\#_\chi$ takes the form
$\tilde\Delta^\#_\chi=\tilde\Delta\otimes \Id_{V_\chi}$, where $\tilde\Delta$
is the Laplace operator on  $S$. 
Let $h_\varphi\in C_c^\infty(G//K)$ be the kernel of the $G$-invariant integral 
operator $\varphi\left(\tilde\Delta^{1/2}\right)$.
Then it follows that the kernel $K_\varphi(x,y)$
of $\varphi\left((\Delta_\chi^\#)^{1/2}\right)$ is given by
\begin{equation}\label{kernel5}
K_\varphi(x,y)=\sum_{\gamma\in\Gamma} h_\varphi(g_1^{-1}\gamma g_2)\chi(\gamma),
\end{equation}
where $x=\Gamma g_1K$ and $y=\Gamma g_2 K$. One can now proceed in the same 
way as in the case of a unitary representation $\chi$ and derive the twisted 
Selberg trace formula.

Besides unitary representations of $\Gamma$,  there is a second class of
representations of $\Gamma$ for which the usual trace formula can be applied.
These are representations which are the restriction to $\Gamma$ of a
finite-dimensional representation $\rep\colon G\to\GL(E)$. Let $E_\rep\to X$
be the flat vector bundle associated to $\rep|_\Gamma$. Then $E_\rep$ is
canonically isomorphic to the locally homogeneous vector bundle $E_\tau$ 
associated to the principal $K$-bundle $\Gamma\bs G\to X$ via the 
representation $\tau=\rep|_K$. The bundle carries a canonical Hermitian fiber 
metric
and the Laplacian in $C^\infty(X,E_\rep)$ with respect to this metric is closely
related to the Casimir operator acting in $C^\infty(X,E_\tau)$. This brings us
back to the usual framework of the Selberg trace formula for locally 
homogeneous vector bundles. Details will be discussed in section 
\ref{restrict}.

The paper is organized as follows. In section \ref{funccal} we collect a number
of facts about spectral theory of elliptic operators with leading symbol of
Laplace type and we develop some functional calculus for such operators.
The kernels of the associated integral operators are studied in section 
\ref{kernwave}.  Especially, we prove (\ref{intoper}) and 
(\ref{kernel5}). In section \ref{bochlaplace} we apply these results to the 
case of twisted Bochner-Laplace operators. In section \ref{locsym}
we turn to the locally symmetric case 
and we prove the first version of the trace formula which is Theorem 
\ref{traceform}. In  section \ref{rank1} we specialize to
the case where $G$ has
split rank one and we prove Theorem \ref{scalar}. In the final section 
\ref{restrict} we are concerned with representations of $\Gamma$ which are 
the restriction of a representation of $G$.

\section{Functional calculus}\label{funccal}
\setcounter{equation}{0}

In this section we develop the necessary facts of the functional calculus we 
are going to use in this paper. 

Let $X$ be a closed Riemannian manifold of dimension $n$ and $E\to X$ a 
Hermitian vector bundle 
over $X$. We denote by $C^\infty(X,E)$ the space of smooth sections of $E$,
and by $L^2(X,E)$ the space of $L^2$-sections of $E$ w.r.t. the metrics on
$X$ and $E$. Let
\[
P\colon C^\infty(X,E)\to C^\infty(X,E)
\]
be an elliptic differential operator of order 2 with leading symbol
\begin{equation}\label{symbol1}
\sigma(P)(x,\xi)=\parallel\xi\parallel^2_x\cdot\Id_{E_x}.
\end{equation}
For $I\subset[0,2\pi]$ let 
\[
\Lambda_I=\bigl\{re^{i\theta}\colon 0\le r<\infty,\; \theta\in I\}.
\]
be the solid angle attached to $I$. The following lemma describes the 
structure of the spectrum of $P$. 
\begin{lem}\label{spec1}
For every $ 0<\varepsilon<\pi/2$ there exists $R>0$ such that the spectrum of
$P$ is contained in the set $B_R(0)\cup \Lambda_{[-\varepsilon,\varepsilon]}$. 
Moreover the spectrum of $P$ is discrete.
\end{lem}
\begin{proof} The first statement follows from \cite[Theorem 9.3]{Sh}. 
The discreteness of the spectrum follows from \cite[Theorem 8.4]{Sh}. 
\end{proof}

It follows from  Lemma \ref{spec1} that there exists an
Agmon angle $\theta$  for $P$ and we can 
define the square root $P_\theta^{1/2}$ as in \cite{Sh}. For the convenience 
of the reader we include some details. Denote by $\spec(P)$ the spectrum of 
$P$. Let $\varepsilon>0$. For simplicity we assume that 
$0\notin\sigma(P)$. 
By Lemma \ref{spec1} there exist $0<\theta<2\pi$ and $\varepsilon>0$ such that
\[
\spec(P)\cap \Lambda_{[\theta-\varepsilon,\theta+\varepsilon]}=\emptyset.
\]
$\theta$ is called an Agmon angle for $P$. Since $\spec(P)$ is discrete
and $0\notin\sigma(P)$, there exists also $r_0>0$ such that
\[
\spec(P)\cap \{z\in\C\colon |z|<2r_0\}=\emptyset.
\]
Define the contour $\Gamma=\Gamma_{\theta,r_0}\subset \C$ as the union of three
curves $\Gamma=\Gamma_1\cup\Gamma_2\cup\Gamma_3$, where
\[
\begin{split}
\Gamma_1=\{re^{i\theta}\colon r_0&\le r<\infty\},\quad \Gamma_2=\{r_0e^{i\alpha}
\colon \theta\le\alpha\le\theta+2\pi\},\\
&\Gamma_3=\{re^{i(\theta+2\pi)}\colon r_0\le r<\infty\}.
\end{split}
\]
Put
\[
P^{-1/2}_\theta=\frac{i}{2\pi}\int_{\Gamma_{\theta,r_0}}
\lambda^{-1/2}(P-\lambda)^{-1}\,d\lambda.
\]
By \cite[Corollary 9.2, Chapt. II, \S 9]{Sh} we have 
$\parallel(P-\lambda)^{-1}\parallel\le C|\lambda|^{-1}$ for 
$\lambda\in\Gamma_{\theta,r_0}$. Therefore the integral is absolutely 
convergent. Put
\[
P^{1/2}_\theta=P\cdot P^{-1/2}_\theta.
\]
Then $P^{1/2}_\theta$ satisfies $(P^{1/2}_\theta)^2=P$. If $\theta$ is fixed, we 
simply denote this operator by $P^{1/2}$. 
We recall \cite{See}, \cite[Theorem 11.2]{Sh} that $P^{1/2}$ is a 
classical pseudo-differential operator with principal symbol 
\[
\sigma(\rP)(x,\xi)=\parallel\xi\parallel_x\cdot\Id_{E_x}.
\]
In any coordinate chart, the complete symbol $q(x,\xi)$ of $P^{1/2}$
has an asymptotic expansion
\[
q(x,\xi)\sim\sum_{j=0}^\infty q_{1-j}(x,\xi),\]
where $q_{1-j}(x,\xi)$ is homogeneous in $\xi$ of order $1-j$. The same holds 
for $\rD$. Since the principal symbols coincide, we get
\begin{equation}\label{sqroot}
\rP=\rD+B,
\end{equation}
where $B$ is a pseudo-differential operator of order zero. Especially, $B$
is a bounded operator in $L^2(X,E)$. 

Let $R_\lambda(\rP)=(\rP-\lambda\I)^{-1}$ and $R_\lambda(\rD)
=(\rD-\lambda\I)^{-1}$ be the resolvents of $\rP$ and $\rD$, respectively. For
$\lambda\not\in\spec(\rD)$ we have the following equality
\begin{equation}\label{resolv1}
\rP-\lambda\I=(\I+BR_\lambda(\rD))(\rD-\lambda\I).
\end{equation}
Since $\rD$ is self-adjoint, the resolvent of $\rD$ satisfies
\begin{equation}\label{resolv2}
\parallel R_\lambda(\rD)\parallel\le |\Im(\lambda)|^{-1}
\end{equation}
\cite[Chapt. V, \S3.5]{Ka}. Let $b=2\parallel B \parallel$.  It follows
from (\ref{resolv2}) that  for $|\Im(\lambda)|\ge b$ we have
$\parallel BR_\lambda(\rD)\parallel\le 1/2$. Thus in this range of $\lambda$
the operator $\I+BR_\lambda(\rD)$ is invertible and 
\[
R_\lambda(\rP)=R_\lambda(\rD)(\I+BR_\lambda(\rD))^{-1}.
\]
Combined with (\ref{resolv2}) we get
\begin{equation}\label{resolv3}
\parallel R_\lambda(\rP)\parallel \le 2|\Im(\lambda)|^{-1},\quad 
|\Im(\lambda)|\ge b.
\end{equation}

We can now summarize the spectral properties of $\rP$.
\begin{lem}\label{spec2}
The resolvent of $\rP$ is compact.
The spectrum of $\rP$ is discrete. There exist $b>0$ and $c\in\R$
such that the spectrum of $\rP$ is contained in the domain
\begin{equation}
\Omega_{b,c}=\bigl\{\lambda\in\C\colon \Re(\lambda)>c,\;
\;|\Im(\lambda)|<b\bigr\}.
\end{equation}\label{domain}
\end{lem}
\begin{proof}
Since $\rP$ is an elliptic pseudo-differential operator of order 1 on a closed 
manifold,
its resolvent is compact and hence, its spectrum is discrete. The remaining 
statements are a consequence of 
(\ref{resolv3}).
\end{proof}

Though $P$ is not self-adjoint in general, it still has nice spectral 
properties \cite[Chapt. I, \S 8]{Sh}. Given $\lambda_0\in\spec(P)$, let
$\Gamma_{\lambda_0}$ be a small circle around $\lambda_0$ which
contains no other points of $\spec(P)$. Put
\[
\Pi_{\lambda_0}=\frac{i}{2\pi}\int_{\Gamma_{\lambda_0}}R_\lambda(P)\;d\lambda.
\]
Then $\Pi_{\lambda_0}$ is the projection onto the root subspace $V_{\lambda_0}$.
This is a finite-dimensional subspace of $C^\infty(X,E)$ which
is invariant under $P$ and there exists $N\in\N$ such that 
$(P-\lambda_0\I)^NV_{\lambda_0}=0$. Furthermore, there is a closed complementary
subspace $V_{\lambda_0}^\prime$ to $V_{\lambda_0}$ in $L^2(X,E)$ 
which is invariant under the 
closure $\bar P$ of $P$ in $L^2$ and the restriction of $(\bar P-\lambda_0\I)$ 
to $V_{\lambda_0}^\prime$ has a bounded inverse. The {\it algebraic multiplicity}
$m(\lambda_0)$ of $\lambda_0$ is defined as
\[
m(\lambda_0)=\dim V_{\lambda_0}.
\]
If $\lambda_1,\lambda_2\in
\spec(P)$ with $\lambda_1\not=\lambda_2$, then the projections $\Pi_{\lambda_1}$
and $\Pi_{\lambda_2}$ are disjoint, i.e., 
\[
\Pi_{\lambda_1}\Pi_{\lambda_2}=\Pi_{\lambda_2}\Pi_{\lambda_1}=0.
\]
Since the principal symbols of $P$ and $\Delta_E$ are the same,
we have $P=\Delta+D$,
where $D$ is a first order differential operator. Let $R_\lambda(\Delta)$ be 
the resolvent of $\Delta$. Then $DR_\lambda(\Delta)$ is a compact operator. 
This means that $D$ is compact relative to $\Delta$ and it follows from
\cite[I,\S 4, Theorem 4.3]{Mk} that the root vectors are complete. Thus
$L^2(X,E)$ is the closure of the algebraic direct sum of finite-dimensional 
$P$-invariant subspaces $V_k$
\begin{equation}\label{directsum}
L^2(X,E)=\overline{\bigoplus_{k\ge 1} V_k}
\end{equation}
such that the restriction of $P$ to $V_k$ has a unique eigenvalue $\lambda_k$
and $|\lambda_k|\to\infty$. In general, the sum (\ref{directsum}) is not a sum
of mutually orthogonal subspaces.

It follows from (\ref{directsum}) that 
$P^{1/2}_\theta$ has a similar spectral decomposition with eigenvalues 
$\lambda^{1/2}_\theta$, $\lambda\in\spec(P)$, and 
$m(\lambda^{1/2}_\theta)=m(\lambda)$.

Given $r>0$, let 
\[
N(r,P):=\sum_{\lambda\in\spec(P),\;|\lambda|\le r}m(\lambda).
\]
be the counting function of the eigenvalues of $P$, where eigenvalues 
are counted with their algebraic multiplicity. 
\begin{lem}\label{counting} 
Let $n=\dim X$. We have
\[
N(r,P)=\frac{\rkk(E)\vol(X)}{(4\pi)^{n/2}\Gamma(n/2+1)} r^{n/2}+o(r^{n/2}),
\quad r\to\infty.
\]
\end{lem}
\begin{proof}
It is well know that $\Tr(e^{-t\Delta_E})$ has an asymptotic 
expansion of the form
\[
\Tr(e^{-t\Delta_E})\sim t^{-n/2}\sum_{k\ge 0}a_kt^k,\quad t\to +0
\]
(see \cite[Lemma 1.8.3]{Gi}), and by \cite[Lemma 4.1.4]{Gi}
the leading coefficient $a_0$ 
is given by $a_0=(4\pi)^{-n/2}\rkk(E)\vol(X)$.
Let $N(r,\Delta_E)$ be the counting function of the eigenvalues of $\Delta_E$. 
Using the Tauberian theorem (see \cite[Chapt. II, \S 14]{Sh}), we get
\begin{equation}\label{weyllaw}
N(r,\Delta_E)=\frac{\rkk(E)\vol(X)}{(4\pi)^{n/2}\Gamma(n/2+1)}  
r^{n/2}+o(r^{n/2}),\quad r\to\infty.
\end{equation}
The lemma follows from \cite[I, \S 8, Corollary 8.5]{Mk}.
\end{proof}
Let $h\in C_c^\infty(\R)$ be even and set
\[
\varphi(\lambda)=\frac{1}{\sqrt{2\pi}}\int_\R h(r)\cos(t\lambda)\,dr,\quad 
\lambda\in\C.
\]
Then $\varphi(\lambda)$ is rapidly decreasing in each strip 
$|\Im(\lambda)|<\delta$, $\delta>0$. For $b>0$ and $d\in\R$ let 
$\Gamma=\Gamma_{b,d}\subset \C$ be the contour which is union of the two 
half-lines $L_{\pm b,d}=\{z\in\C\colon \Im(z)=\pm b,\;\Re(z)\ge d\}$ and the 
semi-circle $S=\{d+be^{i\theta}\colon \pi/2\le\theta\le 3\pi/2\}$.
By Lemma \ref{spec2} there exist $b>0$, $d\in\R$ such that 
$\spec(\rP)$ is contained in the interior of $\Gamma_{b,d}$. Put
\begin{equation}\label{integral1}
\varphi(\rP):=\frac{i}{2\pi}
\int_\Gamma\varphi(\lambda)(\rP-\lambda\I)^{-1}\,d\lambda.
\end{equation}
It follows from (\ref{resolv3}) that the integral is absolutely convergent. 

\begin{lem}\label{smoothing}
$\varphi(\rP)$ is an integral operator with a smooth kernel.
\end{lem}
\begin{proof}
First note that
\[
\int_\Gamma\varphi(\lambda)(P-\lambda^2)
(\rP-\lambda)^{-1}\,d\lambda=\int_\Gamma\varphi(\lambda)
(\rP+\lambda)\,d\lambda=0.
\]
This implies that for $k,l\in\N$ we have
\[
P^k\varphi(\rP)P^l=\frac{i}{2\pi}
\int_\Gamma\lambda^{2(k+l)}\varphi(\lambda)(\rP-\lambda)^{-1}\,d\lambda.
\]
The function $\lambda\mapsto\lambda^{2(k+l)}\varphi(\lambda)$ is rapidly
decreasing on $|\Im(\lambda)|=\pm b$. Hence $P^k\varphi(\rP)P^l$ is
a bounded operator in $L^2$. Since $P$ is elliptic, it follows that for all
$s,r\in\R$, 
$\varphi(\rP)$ extends to a bounded operator from $H^s(X,E)$
to $H^r(X,E)$, which shows that $\varphi(\rP)$ is a smoothing operator and
hence, it is an integral operator with a smooth kernel.
\end{proof}

In order to continue, we need to establish an auxiliary result about smoothing
operators. Let 
\[
A\colon L^2(X,E)\to L^2(X,E)
\]
be an integral operator with a smooth kernel $H\in C^\infty(X\times X,E\boxtimes
E^*)$. 
\begin{prop}\label{trace} $A$ is a trace class operator and
\begin{equation}\label{matrixtr}
\Tr(A)=\int_X\tr H(x,x)\;d\mu(x).
\end{equation}
\end{prop}
\begin{proof} We generalize the proof of Theorem 1 in 
\cite[Chapt VII, \S 1]{La}.
Let $\nabla^E$ be a Hermitian connection in $E$ and let $\Delta_E
=(\nabla^E)^*\nabla^E$
be the associated Bochner-Laplace operator. Then $\Delta_E$ is a second order 
elliptic operator which is essentially self-adjoint and non-negative. Its
spectrum is discrete. 
Let $\{\phi_j\}_{j\in\N}$ be an
orthonormal basis of $L^2(X,E)$ consisting of eigensections of $\Delta_E$ with
eigenvalues $0\le\lambda_1\le\lambda_2\le\cdots\to \infty $.  
We can expand $H$ in the orthonormal basis as
\begin{equation}\label{expans}
H(x,y)=\sum_{i,j=1}^\infty a_{i,j}\phi_i(x)\otimes\phi_j^*(y),
\end{equation}
where 
\begin{equation}\label{coeffic}
a_{i,j}=\langle A\phi_i,\phi_j\rangle.
\end{equation}
Since $H$ is smooth, the coefficients $a_{i,j}$ are rapidly decreasing. Indeed,
for every $N$ we have
\[
(1+\lambda_i+\lambda_j)^Na_{i,j}=\langle (\I+\Delta_E\otimes\I
+\I\otimes\Delta_E)^N H,\phi_i\otimes\phi_j^*\rangle.
\]
Hence for every $N\in\N$ there exists $C_N>0$ such that 
\[
|a_{i,j}|\le C_N(1+\lambda_i+\lambda_j)^{-N},\quad i,j\in\N.
\]
This implies that the series (\ref{expans}) converges in the $C^\infty$ 
topology. Let 
$P_{i,j}$ be the integral operator with kernel $\phi_i\otimes \phi_j^*$. Thus
\[
P_{i,j}(\phi_k)=\begin{cases}0,&\;k\not=j;\\\phi_i,&\;k=j.
\end{cases}
\]
Let $P_j$ be the orthogonal projection of $L^2(X,E)$ onto the 1-dimensional
subspace $\C\phi_j$. Put
\[
B=\sum_{i,j}a_{i,j}(1+\lambda_j)^{n}P_{i,j},\quad C=\sum_{j}(1+\lambda_j)^{-n}
P_j.
\]
Then $A=BC$ and it follows from (\ref{weyllaw}) that $B$ and $C$ are 
Hilbert-Schmidt operators. Thus $A$ is a 
trace class operator. Furthermore, by (\ref{expans}) and (\ref{coeffic}) we get
\[
\int_X\tr H(x,x)\;dx=\sum_{i,j=1}^\infty a_{i,j}\int_X\langle\phi_i(x),\phi_j(x)
\rangle\; dx=\sum_{i=1}^\infty a_{i,i}=\Tr(A).
\]
\end{proof}

Now we apply this result to $\varphi(\rP)$. Let $K_\varphi(x,y)$ be the 
kernel of $\varphi(\rP)$. Then by Proposition \ref{trace}, $\varphi(\rP)$
is a trace class operator and we have
\begin{equation}\label{matrixtr1}
\Tr\varphi(\rP)=\int_X\tr K_\varphi(x,x)\;d\mu(x).
\end{equation}
By Lidskii's theorem \cite[Theorem 8.4]{GK} the trace is
equal to the sum of the eigenvalues of $\varphi(\rP)$, counted with their 
algebraic multiplicities. The eigenvalues of $\varphi(\rP)$ and their algebraic
multiplicities can be determined as follows. Given $N\in\N$, 
let $\Pi_N$ denote the projection onto the direct
sum of the root subspaces $V_k$, $k\le N$, of $P$. As explained above, we have
\[
P\Pi_N=\sum_{k=1}^N(\lambda_k\Pi_k+D_k),
\]
where $\Pi_k$ is the projection onto $V_k$ and $D_k$ is a nilpotent operator
in $V_k$. Then it follows from \cite[I, (5.50)]{Ka} that 
\[
\varphi(\rP)\Pi_N=\sum_{k=1}^N(\varphi(\lambda^{1/2}_k)\Pi_k+D_k^\prime),
\]
where $D_k^\prime$ is again a nilpotent operator in $V_k$. Thus $\varphi(\rP)$ 
leaves the decomposition (\ref{directsum}) invariant and the restriction of
$\varphi(\rP)$ to $V_k$ has a unique eigenvalue $\varphi(\lambda^{1/2}_k)$.
Of course, some of the eigenvalues $\varphi(\lambda^{1/2}_k)$ may coincide in
which case the root space is the sum of the corresponding root spaces $V_k$.
Now, applying Lidskii's theorem \cite[Theorem 8.4]{GK} and (\ref{matrixtr1}),
we get the following proposition.
\begin{prop}\label{trace1} Let $\varphi\in\cS(\R)$ be even and suppose that 
$\widehat\varphi\in C^\infty_c(\R)$. Then we have
\begin{equation}\label{trace1a}
\sum_{\lambda\in\spec(P)}m(\lambda)\varphi(\lambda^{1/2})=
\int_X\tr K_\varphi(x,x)\;dx.
\end{equation}
\end{prop}
By Lemma \ref{counting}, the series on the left hand side is absolutely 
convergent. 

\noindent
{\bf Remark.} Recall that $\rP=\rP_\theta$ depends on
the choice of an Agmon angle $\theta$, and so do the eigenvalues
$\lambda_k^{1/2}=(\lambda_k)_\theta^{1/2}$. Let $0<\theta<\theta^\prime<2\pi$
be two Agmon angles. Then it follows from Lemma \ref{spec1} that there are
only finitely many eigenvalues $\lambda_{1},...,\lambda_{m}$ of $P$ which 
are contained in $\Lambda_{[\theta,\theta^\prime]}$. Therefore for $\lambda\in
\spec(P)$ we have
\[
(\lambda)_{\theta^{\prime}}^{1/2}=\begin{cases}\hskip9pt (\lambda)_\theta^{1/2},& 
\mathrm{if}\;\lambda\not\in\{\lambda_1,...,\lambda_m\};\\
-(\lambda)_\theta^{1/2},& \mathrm{if}\;\lambda\in\{\lambda_1,...,\lambda_m\}.
\end{cases}
\]
Since $\varphi$ is even $\varphi\left((\lambda)_\theta^{1/2}\right)$ is 
independent of $\theta$. This justifies the notation on the left hand side
of (\ref{trace1a}).
\hfill$\square$

\section{The kernel and the wave equation}\label{kernwave}
\setcounter{equation}{0}

In this section we give a description of the kernel $K_\varphi$ of the 
smoothing operator $\varphi(\rP)$
in terms of the solution of the wave equation. Consider the wave equation
\begin{equation}\label{waveequ}
\frac{\partial^2 u}{\partial t^2}+P u=0,\quad u(0,x)=f(x),\; 
u_t(0,x)=0.
\end{equation}
\begin{prop}\label{uniqueness}
For each $f\in C^\infty(X,E)$ there is a unique solution $u(t;f)\in 
C^\infty(\R\times X,E)$ 
of the wave equation (\ref{waveequ}) with initial condition $f$. 
Moreover for every $T>0$ and $s\in\R$ there exists $C>0$ such that for
every $f\in C^\infty(X,E)$
\begin{equation}\label{bound1}
\parallel u(t,f)\parallel_s\le C\parallel u(0,f)\parallel_s,\quad |t|\le T,
\end{equation} 
where $\parallel\cdot\parallel_s$ denotes the $s$-Sobolev norm.
\end{prop}
\begin{proof}
We proceed in the same way as in  \cite[Chapt. IV, \S\S 1,2]{Ta1} and replace 
(\ref{waveequ}) by a first oder system. Let $\Delta_E$ be the Bochner-Laplace
operator associated to the connection $\nabla^E$ in $E$. Put
$\Lambda=(\Delta_E+\Id)^{1/2}$ and
\begin{equation}\label{order1}
L:=\begin{pmatrix} 0&\Lambda\\ -P\Lambda^{-1}&0\end{pmatrix}\colon 
\begin{matrix} C^\infty(X,E)\\\oplus\\C^\infty(X,E)\end{matrix}
\longrightarrow \begin{matrix}C^\infty(X,E)\\\oplus\\C^\infty(X,E)\end{matrix}.
\end{equation}
Then $L$ is a pseudo-differential operator of order 1. Let $u$ be a solution of
(\ref{waveequ}). Put 
\[
u_1=\Lambda u,\quad u_2=\frac{\partial}{\partial t}u.
\] 
Then $(u_1,u_2)$ satisfies
\begin{equation}\label{system}
\frac{\partial}{\partial t}\begin{pmatrix}u_1\\ u_2\end{pmatrix}=
L\begin{pmatrix}u_1\\ u_2\end{pmatrix},\quad u_1(0)=\Lambda f,\; u_2(0)=0.
\end{equation}
On the other hand, let $(u_1,u_2)$ be a solution of the initial value problem
(\ref{system}). Put $u=\Lambda^{-1}u_1$. Then $u$ is a solution of 
(\ref{waveequ}). Thus it suffices to consider (\ref{system}). 
By (\ref{symbol1}) it follows that $P=\Delta_E+D$ where $D$ is a differential
operator of order $\le 1$. Therefore we get
\[
P\Lambda^{-1}=(\Delta_E+\Id)\Lambda^{-1}+(D-\Id)\Lambda^{-1}=\Lambda+B_1,
\] 
where $B_1$ is a pseudo-dfifferential operator of order 0. Therefore
\[
L+L^*=\begin{pmatrix}0&\Lambda\\-\Lambda-B_1&0\end{pmatrix}+
\begin{pmatrix}0&-\Lambda-B_1^*\\\Lambda &0\end{pmatrix}
=-\begin{pmatrix}0&B_1^*\\ B_1&0\end{pmatrix},
\]
is a pseudo-differential operator of order zero. Hence (\ref{system}) is a
symmetric hyperbolic system in the sense of \cite[Chapt. IV, \S 2]{Ta1}. So we
can proceed as in the proof of Theorem 2.3 in \cite[Chapt. IV, \S 2]{Ta1}
to establish existence and uniqueness of solutions of (\ref{waveequ}).
The estimation (\ref{bound1})
follows from the proof using Cronwall's inequality.
\end{proof}

\begin{prop}\label{represent}
Let $\varphi\in {\mathcal S}(\R)$ be even such that $\widehat\varphi
\in C^\infty_c(\R)$. Then for every $f\in C^\infty(X,E)$ we have
\[
\varphi(\rP)f=\frac{1}{\sqrt{2\pi}}
\int_\R\widehat\varphi(t)u(t;f)\;dt.
\]
\end{prop}
\begin{proof}
Let $\Gamma\subset \C$ be as in (\ref{integral1}). Let $c\ge 0$ be such that 
the spectrum of
$P+c$ is contained in $\Re(z)>0$. For $\sigma>0$ define the operator
$\cos(t\rP)e^{-\sigma(P+c)}$ by the functional integral
\[
\cos(t\rP)e^{-\sigma(P+c)}=\frac{i}{2\pi}
\int_\Gamma \cos(t\lambda) e^{-\sigma(\lambda^2+c)}(\rP-\lambda)^{-1}\, 
d\lambda.
\]
By (\ref{resolv3}) the integral is absolutely convergent.
For $f\in C^\infty(X,E)$ and $\sigma>0$ put
\begin{equation}\label{solution}
u(t;\sigma,f):=\cos(t\rP)e^{-\sigma(P+c)}f.
\end{equation}
Then $u(t;\sigma,f)$ satisfies
\[
\begin{split}
\left(\frac{\partial^2}{\partial t^2}+P\right)u(t;\sigma,f)&=\frac{i}{2\pi}
\int_\Gamma \cos(t\lambda) e^{-\sigma(\lambda^2+c)}
(P-\lambda^2)(\rP-\lambda)^{-1}\,d\lambda\\
&=\frac{i}{2\pi}
\int_\Gamma \cos(t\lambda) e^{-\sigma(\lambda^2+c)}(\rP+\lambda)\; d\lambda=0.
\end{split}
\]
and $u(0;\sigma,f)=e^{-\sigma(P+c)}f$. Thus $u(t;\sigma,f)$ is the unique 
solution of (\ref{waveequ}) with initial condition $e^{-\sigma(P+c)}f$. Then
$u(t;f)-u(t;\sigma,f)$ is the solution of (\ref{waveequ}) with initial 
condition $f-e^{-(P+c)}f$. Hence by (\ref{bound1}) we get for all $s\in\R$
\begin{equation}\label{bound2}
\parallel u(t;f)-u(t;\sigma,f)\parallel_{H^s}\le C\parallel f-e^{-\sigma(P+c)}f
\parallel_{H^s},\quad |t|\le T.
\end{equation}
 Now note that for every $f\in C^\infty(X,E)$ we have
\[
\lim_{\sigma\to 0} \parallel e^{-\sigma(P+c)}f-f\parallel=0.
\]
This follows from the parametrix construction. Hence we get
\[
\begin{split}
\parallel f-e^{-\sigma(P+c)}f \parallel_{H^s}&=\;\parallel (P+c)^{s/2}f- 
e^{-\sigma(P+c)}(P+c)^{s/2}f\parallel_{L^2}\to 0
\end{split}
\]
as $\sigma\to 0$. Combined with (\ref{bound2}) we get
\begin{equation}\label{limit1}
\lim_{\sigma\to 0}\parallel u(t;f)-u(t;\sigma,f)\parallel_{H^s}=0.
\end{equation}
Furthermore we have
\[
\begin{split}
\frac{1}{\sqrt{2\pi}}\int_\R\widehat\varphi(t)u(t;\sigma,f)\;dt&
=\frac{1}{\sqrt{2\pi}}\int_\R\widehat\varphi(t)
\frac{i}{2\pi}
\int_\Gamma \cos(t\lambda)e^{-\sigma(\lambda^2+c)}
(\rP-\lambda)^{-1}f\,d\lambda\;dt\\
&=\frac{i}{2\pi}
\int_\Gamma\left(\frac{1}{\sqrt{2\pi}}\int_\R\widehat\varphi(t)
\cos(t\lambda)\;dt\right)
e^{-\sigma(\lambda^2+c)}(\rP-\lambda)^{-1}f\,d\lambda\\
&=\frac{i}{2\pi}\int_\Gamma\varphi(\lambda)e^{-\sigma(\lambda^2+c)}
(\rP-\lambda)^{-1}f\,d\lambda.
\end{split}
\]
For $\sigma\to 0$, the right hand side converges to $\varphi(\rP)f$. By
(\ref{limit1}) the left hand side converges to $(2\pi)^{-1/2}\int_\R 
\widehat\varphi(t)u(t;f)\;dt.$ 
\end{proof}

Let $p\colon \widetilde X\to X$ be the universal covering of $X$, $\widetilde
E=p^*E$, and 
$\widetilde P\colon C^\infty(\widetilde X,\widetilde E)\to 
C^\infty(\widetilde X,\widetilde E)$ the lift of $P$ to $\widetilde X$. Let
$\widetilde u(t,\tilde x,f)$ and $\widetilde f$ be the pull back 
 to $\widetilde X$ of $u(t,x;f)$ and $f$, respectively. 
Then $\widetilde u(t,f)$ satisfies
\begin{equation}\label{waveequ1}
\left(\frac{\partial^2 }{\partial t^2}+\widetilde P\right)\widetilde 
u(t;f)=0,\quad \widetilde u(0;f)=\widetilde f,\; \widetilde u_t(0,f)=0.
\end{equation}
By (\ref{symbol1}) we have  $\widetilde P=\widetilde \Delta_E+\widetilde D$,
where $\widetilde D$ is a differential operator of order $\le 1$. Then it
follows from energy estimates as in \cite[Chapt. 2, \S 8]{Ta2} that solutions 
of $(\frac{\partial^2 }{\partial t^2}+\widetilde P)u=0$  have finite 
propagation speed. This implies
that for every $\psi\in C^\infty(\widetilde X,\widetilde E)$ the
wave equation
\[
\left(\frac{\partial^2 }{\partial t^2}+\widetilde P\right)u(t;\psi)=0,\quad 
u(0;\psi)=\psi,\;u_t(0;\psi)=0,
\]
has a unique solution. Hence we get
\begin{equation}\label{covering}
\widetilde u(t;f)=u(t,\widetilde f).
\end{equation}
Let $d(x,y)$ denote the geodesic distance of $x,y\in\widetilde X$. 
For $\delta>0$ let
\[
U_\delta=\{(x,y)\in \widetilde X\times \widetilde X\colon d(x,y)<\delta\}.
\]
\begin{prop}\label{covering1}
There exist $\delta>0$ and $H_\varphi\in C^\infty(\widetilde X\times
\widetilde X,\Hom(\widetilde E,\widetilde E))$ with $\supp 
H_\varphi\subset U_\delta$ such that for all $\psi\in C^\infty(
\widetilde X,\widetilde E)$ we have
\[
\frac{1}{\sqrt{2\pi}}\int_\R \widehat\varphi(t)u(t,\widetilde x;\psi)\; dt=
\int_{\widetilde X}H_\varphi(\widetilde x,\widetilde y)(
\psi(\widetilde y))\;d\widetilde y.
\]
\end{prop}
\begin{proof}
Suppose that $\supp\widehat\varphi\subset [-T,T]$. Let $V\subset \widetilde X$
be an open relatively compact subset. For $r>0$ let
\[
V_r=\{y\in V\colon d(y,V)<r\}.
\]
Let $\chi\in C^\infty_c(V_{2T})$ such that $\chi=1$ on $V_T$. By finite 
propagation speed, we have
\[
u(t,\widetilde x;\psi)=u(t,\widetilde x;\chi\psi),\quad \widetilde x\in V,
\;|t|<T,\]
for all $\psi\in C^\infty(\widetilde X,\widetilde E)$. Thus we are
 reduced to the case of a compact manifold and the proof follows from
Lemma \ref{smoothing} and Proposition \ref{represent}.
\end{proof}
Using Proposition \ref{covering1} together with (\ref{covering}) and 
Proposition \ref{represent}, we obtain
\begin{equation}\label{covering2}
\varphi(\rP)f(\widetilde x)=\int_{\widetilde X} 
H_\varphi(\widetilde x,\widetilde y)
(\widetilde f(\widetilde y))\;d\widetilde y
\end{equation}
for all $f\in C^\infty(X,E)$. Let $F\subset \widetilde X$ be a fundamental
domain for $\Gamma$. Given $\gamma\in\Gamma$, let $R_\gamma \colon \widetilde E
\to \widetilde E$ be the induced bundle map. Thus for each $\widetilde y\in
\widetilde X$, we have a linear isomorphism  
$R_{\gamma}\colon \widetilde 
E_{\widetilde y}\to \widetilde E_{\gamma\widetilde y}$. Note that $\widetilde f$
satisfies
\[
\widetilde f(\gamma\widetilde y)=R_\gamma(\widetilde f(\widetilde y)),
\quad \gamma\in\Gamma.
\]
Then we get
\[
\begin{split}
\int_{\widetilde X} H_\varphi(\widetilde x,\widetilde y)(
\widetilde f(\widetilde y))\;d\widetilde y&=
\sum_{\gamma\in\Gamma}\int_{\gamma F}
H_\varphi(\widetilde x,\widetilde y)(\widetilde f(\widetilde y))
\;d\widetilde y
\\
&=\sum_{\gamma\in\Gamma}\int_{F}
H_\varphi(\widetilde x,\gamma\widetilde y)(\widetilde 
f(\gamma\widetilde y))\;d\widetilde y\\
&=\int_{F}\left(\sum_{\gamma\in\Gamma}
H_\varphi(\widetilde x,\gamma\widetilde y)\circ R_\gamma\right)
(\widetilde f(\widetilde y))\;d\widetilde y.
\end{split}
\]
Combining this expression with (\ref{covering2}), it follows that the kernel 
$K_\varphi$ of $\varphi(\rP)$ is given by
\begin{equation}\label{kernel}
K_\varphi( x, y)=\sum_{\gamma\in\Gamma}
H_\varphi(\widetilde x,\gamma\widetilde y)\circ R_\gamma,
\end{equation}
where $\widetilde x$ and $\widetilde y$ are any lifts of $x$ and $y$ to the
fundamental domain $F$. So by Proposition \ref{trace1} we get
\begin{prop}\label{trace2}
Let $\varphi\in\cS(\R)$ be even and suppose that 
$\widehat\varphi\in C^\infty_c(\R)$. Then we have
\begin{equation*}
\sum_{\lambda\in\spec(P)}m(\lambda)\varphi(\lambda^{1/2})=\sum_{\gamma\in\Gamma}
\int_F\tr(H_\varphi(\widetilde x,\gamma\widetilde x)\circ R_\gamma)\;
d\widetilde x.
\end{equation*}
\end{prop}
Note that the sum on the right is finite.

\section{The twisted Bochner-Laplace operator}\label{bochlaplace}
\setcounter{equation}{0}

Let $E\to X$ be a complex vector bundle with covariant derivative $\nabla$.
Define the invariant second covariant derivative $\nabla^2$ by
\[
\nabla^2_{U,V}=\nabla_U\nabla_V-\nabla_{\nabla_UV},
\]
where $U,V$ are any two vector fields on $X$. Then the connection Laplacian
$\Delta^\#$ is defined by
\[
\Delta^\#=-\Tr(\nabla^2).
\]
Let $(e_1,...,e_n)$ be a local frame field. Then
\[
\Delta^\#=-\sum_j\nabla^2_{e_j,e_j}.
\]
This formula implies that the principal symbol of $\Delta^\#$ is given by
\[
\sigma(\Delta^\#)(x,\xi)=\parallel \xi\parallel^2_x\Id_{E_x}.
\]
Thus the results of the previous section can be applied to $\Delta^\#$.

Assume that $E$ is equipped with a Hermitian fiber metric and $\nabla$
is compatible with the metric. Then it follows that
\begin{equation}\label{bochlapl}
\nabla^*\nabla=-\Tr\nabla^2,
\end{equation}
\cite[p.154]{LM}, i.e., the connection Laplacian equals the Bochner-Laplace operator
$\Delta_E=\nabla^*\nabla$. 

Now let $\rho\colon \pi_1(X)\to \GL(V)$ be a finite-dimensional complex
representation of $\pi_1(X)$. Let $F\to X$ be the associated flat vector
bundle with connection $\nabla^F$. Let $E$ be a Hermitian vector bundle 
with Hermitian  connection $\nabla^E$. We equip $E\otimes F$ with the
product connection $\nabla^{E\otimes F}$, which is defined by
\[
\nabla^{E\otimes F}_Y=\nabla^E_Y\otimes 1+ 1\otimes \nabla_Y^F, 
\]
for $Y\in C^\infty(X,TX)$. Let $\Delta_{E,\rho}^\#$ be the connection Laplacian
associated to $\nabla^{E\otimes F}$. Locally it can be described as follows.
Let $U\subset X$ be an open subset such that $F|_U$ is trivial. Then 
$(E\otimes F)|_U$ is isomorphic to the direct sum of $m=\rk(F)$ copies of 
$E|_U$:
\[
(E\otimes F)|_U\cong \oplus_{i=1}^m E|_U.
\]
 Let
$e_1,...,e_m$ be a basis of flat sections of $F|_U$. Then each 
$\varphi\in C^\infty(U,(E\otimes F)|_U)$ can be written as
\[
\varphi=\sum_{j=1}^m\varphi_j\otimes e_j,
\]
where $\varphi_i\in C^\infty(U,E|_U)$, $i=1,...,m$. Then 
\[
\nabla^{E\otimes F}_Y(\varphi)=\sum_j(\nabla^E_Y\varphi_j)\otimes e_j.
\]
Let $\Delta_E=(\nabla^E)^*\nabla^E$ be the Bochner-Laplace operator associated
to $\nabla^E$. Using (\ref{bochlapl}), we get 
\begin{equation}\label{twistlapl}
\Delta_{E,\rho}^\#\varphi=\sum_j(\Delta_E\varphi_j)\otimes e_j.
\end{equation}
Let $\widetilde E$ and $\widetilde F$ be the pullback to $\widetilde X$
of $E$ and $F$, respectively. Then $\widetilde F\cong \widetilde X\times V$ and
\[
C^\infty(\widetilde X,\widetilde E\otimes \widetilde F)\cong 
C^\infty(\widetilde X,\widetilde E)\otimes V.
\]
It follows from (\ref{twistlapl}) that with respect to this isomorphism, the
lift $\widetilde \Delta_{E,\rho}^\#$ of $\Delta_{E,\rho}^\#$ to $\widetilde X$ 
takes the form
\[
\widetilde \Delta_{E,\rho}^\#=\widetilde\Delta_E\otimes \Id,
\]
where $\widetilde\Delta_E$ is the lift of $\Delta_E$ to $\widetilde X$. Let 
$\psi\in C^\infty_c(\widetilde X,\widetilde E)\otimes V$. 
Then the unique 
solution of the wave equation 
\[
\left(\frac{\partial^2}{\partial t^2}+\widetilde \Delta^\#_{E,\rho}\right)
u(t;\psi)=0,\quad u(0;\psi)=\psi,\; u_t(0,\psi)=0,
\]
is given by
\[
u(t;\psi)=\left(\cos\left(t(\widetilde \Delta_E)^{1/2}\right)
\otimes\Id\right)\psi.
\]  
Let $\varphi$ be as above and let $k_\varphi(\widetilde x,\widetilde y)$ 
be the kernel of 
\[
\varphi\left((\widetilde \Delta_E)^{1/2}\right)
=\frac{1}{\sqrt{2\pi}}\int_\R \widehat \varphi(t)
\cos \left(t(\widetilde \Delta_E)^{1/2}\right)\;dt.
\]
Then the kernel $H_\varphi$ of Proposition \ref{covering1} is given by
$H_\varphi(\widetilde x,\widetilde y)=k_\varphi(\widetilde x,\widetilde y)
\otimes \Id$. Let $R_{\gamma}\colon \widetilde 
E_{\widetilde y}\to \widetilde E_{\gamma\widetilde y}$ be the canonical
isomorphism. Then it follows from (\ref{kernel}) that the kernel of the
 operator $\varphi\bigl((\Delta^\#_{E,\rho})^{1/2}\bigr)$ is given by
\begin{equation}\label{kernel3}
K_\varphi(x,y)=\sum_{\gamma\in\Gamma} k_\varphi(\widetilde x,\gamma\widetilde y)
\circ(R_\gamma\otimes\rho(\gamma)).
\end{equation}
Combined with (\ref{trace2}) we get
\begin{prop}\label{trace3} Let $F_\rho$ be a flat vector bundle over $X$,
associated to a finite-dimensional complex representation 
$\rho\colon \pi_1(X)\to \GL(V)$. Let $\Delta^\#_{E,\rho}$ be the 
twisted connection Laplacian acting
in $C^\infty(X,E\otimes F_\rho)$. Let $\varphi\in\mathcal{S}(\R)$ be even with 
$\widehat \varphi\in C^\infty_c(\R)$ and denote by 
$k_\varphi(\widetilde x,\widetilde y)$ the kernel of 
$\varphi\left((\widetilde\Delta_{E})^{1/2}\right)$. Then we have
\begin{equation}
\sum_{\lambda\in\spec(\Delta^\#_{E,\rho})} m(\lambda)\varphi(\lambda^{1/2})=
\sum_{\gamma\in\Gamma}\tr\rho(\gamma)
\int_F \tr \left(k_\varphi(\widetilde x,\gamma\widetilde x)
\circ R_\gamma\right)\; d\widetilde x.
\end{equation}
\end{prop}

\section{Locally symmetric spaces}\label{locsym}
\setcounter{equation}{0}

In this section we specialize to the case where $X$ is a locally symmetric 
manifold. Let $G$ be a
connected semisimple real Lie group of non-compact type with finite center.
Let $K\subset G$ be a maximal compact subgroup of $G$. Denote by $\gf$ and 
$\kf$ the Lie algebras of $G$ and $K$, respectively. Let
\begin{equation}\label{cartan}
\gf=\pg\oplus\kf
\end{equation}
be the Cartan decomposition. Put $S=G/K$. This is a Riemannian symmetric 
space of non-positive curvature. The invariant metric is obtained by 
translation of the restriction of the Killing form to $\pg\cong T_e(G/K)$. 
Let $\Gamma\subset G$ 
be a discrete, torsion free, cocompact subgroup. Then $\Gamma$ acts freely
on $S$ by isometries and $X=\Gamma\bs S$ is a compact locally symmetric 
manifold. 

Let $\tau \colon K\to \GL(V_\tau)$ be a finite-dimensional unitary 
representation of $K$, and let
\[
\widetilde E_\tau=(G\times V_\tau)/K\to G/K
\]
be the associated homogeneous vector bundle, where $K$ acts on the right as
usual by 
\[
(g,v)k=(gk,\tau(k^{-1})v),\quad g\in G,\,k\in K,\; v\in V_\tau.
\]
Let 
\begin{equation}
C^\infty(G;\tau):=\left\{f\colon G\to V_\tau\,\,|\,\, f\in C^\infty,\,
f(gk)=\tau(k^{-1})f(g),\,g\in G,\, k\in K\right\}.
\end{equation}
Similarly, by $C_c^\infty(G;\tau)$ we denote the subspace of 
$C^\infty(G;\tau)$ of compactly supported functions and by $L^2(G;\tau)$
the completion of $C_c^\infty(G;\tau)$ with respect to the inner product
\[
\langle f_1,f_2\rangle=\int_{G/K}\langle f_1(g),f_2(g)\rangle\;d\dot g.
\]
There is a canonical isomorphism
\begin{equation}\label{isomor}
C^\infty(S,\widetilde E_\tau)\cong C^\infty(G;\tau).
\end{equation}
\cite[p.4]{Mi}. Similarly, there are isomorphisms 
$C_c^\infty(S,\widetilde E_\tau)\cong C_c^\infty(G;\tau)$
and $L^2(S,\widetilde E_\tau)\cong L^2(G;\tau)$.

Let $\nabla^\tau$ be the canonical 
$G$-invariant connection
on $\widetilde E_\tau$. It is defined by
\[
\nabla^\tau_{g_\ast Y}f(gK)=\frac{d}{dt}\bigg|_{t=0}\left(g\exp(tY)\right)^{-1}
f(g\exp(tY)K),
\]
where $f\in C^\infty(G;\tau)$ and $Y\in\pg$. Let 
$\widetilde \Delta_\tau$ be the associated Bochner-Laplace operator. 
Then $\widetilde \Delta_\tau$ is  $G$-invariant, 
i.e., $\widetilde \Delta_\tau$ commutes with the right action of $G$ on 
$C^\infty(S,\widetilde E_\tau)$. Let $\Omega\in Z(\gf_\C)$ and $\Omega_K\in
Z(\kf_\C)$ be the Casimir elements of $G$ and $K$, respectively. Assume that
$\tau$ is irreducible. Then with respect to (\ref{isomor}), we have
\begin{equation}\label{casimir}
\widetilde\Delta_\tau=-R(\Omega) +\lambda_\tau\Id,
\end{equation}
where $\lambda_\tau=\tau(\Omega_K)$ is the Casimir eigenvalue of $\tau$ 
\cite[Proposition 1.1]{Mi}. We note that $\lambda_\tau\ge 0$.

Let $\varphi\in\cS(\R)$ be even with $\hat\varphi\in C^\infty_c(\R)$. Then
$\varphi(\widetilde\Delta_\tau^{1/2})$ is a $G$-invariant integral operator.
Therefore its kernel $k_\varphi$ satisfies
\[
k_\varphi(g\widetilde x,g\widetilde y)=k_\varphi(\widetilde x,\widetilde y),
\quad g\in G.
\]
With respect to the isomorphism (\ref{isomor}) it can be identified with a
compactly supported $C^\infty$-function
\[
h_\varphi\colon G\to \End(V_\tau),
\]
which satisfies 
\[
h_\varphi(k_1gk_2)=\tau(k_1)\circ h_\varphi(g)\circ \tau(k_2),\quad k_1,k_2\in K.
\]
Then  $\varphi(\widetilde\Delta_\tau^{1/2})$ acts by convolution
\begin{equation}\label{convolu}
\left(\varphi(\widetilde\Delta_\tau^{1/2})f\right)(g_1)
=\int_G h_\varphi(g_1^{-1}g_2)(f(g_2))\,dg_2.
\end{equation}
Let 
\[
E_\tau=\Gamma\bs \widetilde E_\tau
\]
be the locally homogeneous vector bundle over $\Gamma\bs S$ induced by
$\widetilde E_\tau$.
Let $\chi\colon \Gamma\to \GL(V_\chi)$ be a finite-dimensional complex 
representation and let $F_\chi$ be the associated flat vector bundle over
$\Gamma\bs S$. Let $\Delta^\#_{\tau,\chi}$ be the twisted connection Laplacian
acting in $C^\infty(\Gamma\bs S,E_\tau\otimes F_\chi)$. Then it follows from 
(\ref{kernel3}) that the kernel $K_\varphi$ of 
$\varphi(\widetilde\Delta_\tau^{1/2})$ is given by
\[
K_\varphi(g_1K,g_2K)=\sum_{\gamma\in\Gamma}h_\varphi(g_1^{-1}\gamma g_2)\otimes
\chi(\gamma).
\]

By Proposition \ref{trace3} we get
\begin{equation}
\Tr\varphi\left((\Delta^\#_{\tau,\chi})^{1/2}\right)=\sum_{\gamma\in\Gamma}
\tr\chi(\gamma)\int_{\Gamma\bs G}\tr h_\varphi(g^{-1}\gamma g)\,d\dot g.
\end{equation}
We now  proceed in the usual way, grouping terms together into conjugacy 
classes.
Given $\gamma\in\Gamma$, denote by $\{\gamma\}_\Gamma$, $\Gamma_\gamma$, and
$G_\gamma$ the $\Gamma$-conjugacy class of $\gamma$, the centralizer of $\gamma$
in $\Gamma$, and the centralizer of $\gamma$ in $G$, respectively. With the
conjugacy class $\{e\}_\Gamma$ separated from the others as usual, we get a
first version of the trace formula.
\begin{prop}\label{traceform} 
For all even $\varphi\in\cS(\R)$ with $\widehat\varphi\in C_c^\infty(\R)$ we
have
\begin{equation}\label{traceform1}
\begin{split}
\Tr\varphi\left((\Delta^\#_{\tau,\chi})^{1/2}\right)=&
\dim(V_\chi)\vol(\Gamma\bs S)\tr h_\varphi(e)\\
&+\sum_{\{\gamma\}_\Gamma\not=e}\tr\chi(\gamma)\vol(\Gamma_\gamma\bs G_\gamma)
\int_{G_\gamma\bs G}\tr h_\varphi(g^{-1}\gamma g)\;d\dot g.
\end{split}
\end{equation}
\end{prop}

In order to make this formula more explicit, one needs to express the kernel
$h_\varphi$ in terms of $\varphi$, and to evaluate the orbital 
integrals on the right hand side. The kernel $h_\varphi$ can be determined
using Harish-Chandra's Plancherel formula. The orbital integrals can be 
computed using the Fourier inversion formula. However, both formulae are  
pretty complicated in the higher rank case. A sufficiently explicit
formula can be obtained in the rank one case which we discuss in the next 
section.

\section{The rank one case}\label{rank1}
\setcounter{equation}{0}

Let $G$ and $K$ be as above. We introduce some notation following \cite{Wa}.
Let $G=KAN$ be an Iwasawa decomposition of $G$
(see \cite{He}). Then $A$ is a maximal vector subgroup of $G$ and $N$ is
a maximal unipotent subgroup of $G$.  
In this section we assume that $G$ has split rank one, i.e., $\dim A=1$. 
Let $M$ be the
centralizer of $A$ in $K$. We set $P=MAN$. Then $P$ is a parabolic subgroup of 
$G$. Since $G$ has split rank 1, every proper parabolic subgroup of $G$ is
conjugate to $P$. 

Denote by $\widehat G$ and $\widehat M$ the set of 
equivalence classes of irreducible unitary representations of $G$ and $M$,
respectively. For $\pi\in\widehat G$ we denote by $\H_\pi$ the Hilbert space
in which $\pi$ operates.

Let $\af$ and $\nf$ be the Lie algebras of $A$ and $N$, respectively. Choose
$H\in\af$ such that $\ad(H)|_\nf$ has eigenvalues 1 and possibly 2. Then
$\af=\R H$. For $t\in R$ we set $a_t=\exp(tH)$ and $\log a_t= t$. Let 
$A^+=\{a_t\colon t>0\}$.

Let $\rho$ be the half-sum of positive roots of $(\gf,\af)$. Its norm 
$|\rho|$ with 
respect to the normalized Killing form is given as follows. 
Let $p$ and $q$ be the dimensions of the eigenspaces of $\ad(H)|_\nf$ with 
eigenvalues $1$ and $2$, respectively. Then $p>0$ and $0\le q<p$. Then
\begin{equation}\label{casimir1}
|\rho|=\frac{1}{2}(p+2q).
\end{equation}
For $\sigma\in\widehat M$ and $\lambda\in\R$ let $\pi_{\sigma,\lambda}$ be the
unitarily induced representation from $P$ to $G$ which is defined as in 
\cite[p. 177]{Wa}. Let $\Theta_{\sigma,\lambda}$
denote the character of $\pi_{\sigma,\lambda}$. 

If $\gamma\in\Gamma$, $\gamma\not=e$, then there exists $g\in G$ such that
$g\gamma g^{-1}\in MA^+$. Thus there are $m_\gamma\in M$ and $a_\gamma\in A^+$
 such that $g\gamma g^{-1}=m_\gamma a_\gamma$. By \cite[Lemma 6.6]{Wa}, 
$a_\gamma$ depends only on $\gamma$ and $m_\gamma$ is determined by $\gamma$
up to conjugacy in $M$. Let
\[
l(\gamma)=\log a_\gamma.
\]
Then $l(\gamma)$ is the length of the unique closed geodesic of $\Gamma\bs S$ 
determined by $\{\gamma\}_\Gamma$. Furthermore, by the above remark 
\begin{equation}\label{discr}
D(\gamma):=e^{-l(\gamma)|\rho|}\big|\det\left(\Ad(m_\gamma a_\gamma)|_\nf
-\Id\right)\big|
\end{equation}
is well defined. Let
\[
u(\gamma)=\vol(G_{m_\gamma a_\gamma}/A).
\] 
Let $h\in C_c^\infty(G)$ be $K$-finite. Then by \cite[pp. 177-178]{Wa} 
(correcting a misprint) we have
\begin{equation}\label{orbint1}
\int_{G_\gamma\bs G}h(g\gamma g^{-1})\;d\dot g= \frac{1}{2\pi}
\frac{1}{u(\gamma)D(\gamma)}
\sum_{\sigma\in\widehat M}\overline{\tr\sigma(\gamma)}
\int_\R\Theta_{\sigma,\lambda}(h)\cdot e^{-il(\gamma)\lambda}\;d\lambda.
\end{equation}
Since $h$ is $K$-finite, $\Theta_{\sigma,\lambda}(h)\not=0$ only for finitely
many $\sigma$. Thus the sum over $\sigma\in\widehat M$ is finite.
The volume factors in (\ref{traceform1}) are computed as follows. 
Since $G$ has rank one, $\Gamma_\gamma$ is 
infinite cyclic \cite[Proposition 5.16]{DKV}. Thus there is 
$\gamma_0\in\Gamma_\gamma$ such that $\gamma_0$ generates $\Gamma_\gamma$ and
$\gamma=\gamma_0^{n(\gamma)}$ for some integer $n(\gamma)\ge 1$. Then
\begin{equation}\label{volume}
\frac{\vol(\Gamma_\gamma\bs G_\gamma)}{u(\gamma)}=l(\gamma_0).
\end{equation} 
Inserting (\ref{orbint1}) and (\ref{volume}) into 
(\ref{traceform1}), we get the following form of the trace formula
in the rank one case.
\begin{prop}
Let $\varphi\in\cS(\R)$ be even with $\hat\varphi\in C^\infty_c(\R)$. Then
\begin{equation}\label{traceform2}
\begin{split}
\Tr\varphi\left((\Delta^\#_{\tau,\chi})^{1/2}\right)=&
\dim(V_\chi)\vol(\Gamma\bs S)\tr h_\varphi(e)\\
&+\sum_{\{\gamma\}_\Gamma\not=e}\tr\chi(\gamma)\frac{1}{2\pi}
\frac{l(\gamma_0)}{D(\gamma)}
\sum_{\sigma\in\widehat M}\overline{\tr\sigma(\gamma)}
\int_\R\Theta_{\sigma,\lambda}(h_\varphi)\cdot e^{-il(\gamma)\lambda}\;d\lambda.
\end{split}
\end{equation}
\end{prop}
The right hand side is still not in an explicite form. First of all we can
use the Plancherel formula \cite{Kn} to express $\tr h_\varphi(e)$ 
in terms of characters. In this way we are reduced  to the computation of the
characters $\Theta_\pi$, $\pi\in\widehat G$, evaluated on $\tr h_\varphi$. 
This is our next goal.

For simplicity we assume that $K$ is multiplicity free in $G$, i.e., for
each $\tau\in\widehat K$ and $\pi\in \widehat G$, we have $[\pi|_K:\tau]\le 1$.
By \cite{Ko}  this condition is satisfied for $G=SO_0(n,1)$ and $G=SU(n,1)$.
Let 
\[
\widehat G(\tau)=\left\{\pi\in\widehat G\colon [\pi|_K:\tau]=1\right\}.
\]
Then for each $\pi\in\widehat G(\tau)$ we can identify the $\tau$-isotypical
subspace $\H_\pi(\tau)$ of $\tau$ in $\H_\pi$ with $V_\tau$. Let $P_\tau$ be 
the orthogonal 
projection of $\H_\pi$ onto  $\H_\pi(\tau)$. Define the
$\tau$-spherical function $\Phi^\pi_\tau$ on $G$ by
\[
\Phi^\pi_\tau(g):=P_\tau\pi(g) P_\tau,\quad g\in G.
\]
Then $\Phi^\pi_\tau$ is a $C^\infty$-map
\[
\Phi^\pi_\tau\colon G\to \End(V_\tau)
\]
which satisfies
\begin{equation}\label{spherical3}
\begin{aligned}
\Phi^\pi_\tau(g)^*&=\Phi^\pi_\tau(g^{-1})\\
\Phi^\pi_\tau(k_1gk_2)&=\tau(k_1)\Phi^\pi_\tau(g)\tau(k_2), \quad g\in 
G,\; k_1,k_2\in K.
\end{aligned}
\end{equation}
Let $v\in V_\tau$ and set 
\[
f^\pi_{\tau,v}(g)=\Phi^\pi_\tau(g^{-1})(v).
\]
\begin{equation}
\end{equation}
Then $f^\pi_{\tau,v}\in C^\infty(G;\tau)$ and it follows from (\ref{casimir})
that 
\begin{equation}\label{spherical4}
\widetilde\Delta_\tau f^\pi_{\tau,v}=(-\pi(\Omega)+\lambda_\tau)f^\pi_{\tau,v}.
\end{equation}
Let $u(t,x;f^\pi_{\tau,v})$ be the unique solution of
\[
\left(\frac{\partial^2}{\partial t^2}+\widetilde\Delta_\tau\right)u(t)=0,\quad 
u(0)=f^\pi_{\tau,v},\;u_t(0)=0.
\]
\begin{lem}\label{spherical2}
For $t\in\R$, $\tau\in\widehat K$ and $\pi\in\widehat G(\tau)$, we have
$-\pi(\Omega)+\lambda_\tau\ge 0$ and
\[
u(t,x;f^\pi_{\tau,v})=\cos\left(t\sqrt{-\pi(\Omega)+\lambda_\tau}\right)
f^\pi_{\tau,v}(x).
\]
\end{lem}
\begin{proof} Let $\langle\cdot,\cdot\rangle$ be the Killing form on $\gf$. Its
restriction to $\pg$ (resp. $\kf$) is positive (resp. negative) definite.
Let $X_1,...,X_d\in\pg$ and $Y_1,...,Y_m\in\kf$ be bases of $\pg$ and $\kf$, 
respectively, such that $\langle X_i,X_j\rangle =\delta_{ij}$, $\langle Y_i,Y_j
\rangle =-\delta_{ij}$. Then $\Omega=\sum_i X_i^2-\sum_j Y_j^2$ and $\Omega_K=
-\sum_jY_j^2$.  Let $v\in
\H_\pi(\tau)$, $\parallel v\parallel=1$. Then we get
\[
-\pi(\Omega)+\lambda_\tau=
-\langle\pi(\Omega)v,v\rangle+\lambda_\tau =\sum_i\parallel\pi(X_i)v\parallel^2
\ge 0,
\]
which proves the first statement. For the second statement, we note that 
by definition, we have
\begin{equation}\label{partial}
\frac{\partial^2}{\partial t^2}u(t,x;f^\pi_{\tau,v})=-\widetilde\Delta_\tau 
u(t,x;f^\pi_{\tau,v}).
\end{equation}
Fix $x_0\in S$. Let $\chi\in C^\infty_c(S)$ be such that
\[
\chi(y)=\begin{cases}1, & y\in B_{2t}(x);\\0,& y\in S\setminus B_{3t}(x).
\end{cases}
\]
Then by finite propagation speed we have 
\[
u(t,x;f^\pi_{\tau,v})=u(t,x;\chi f^\pi_{\tau,v}),\quad x\in B_t(x_0).
\]
Since $\chi f^\pi_{\tau,v}\in C_c^\infty(S,\widetilde E_\tau)$, we have
\[
u(t,x;f^\pi_{\tau,v})=\left(\cos\left(t(\widetilde\Delta_\tau)^{1/2}\right)
(\chi f^\pi_{\tau,v} )\right)(x),\quad x\in B_t(x_0).
\]
Using that $\widetilde\Delta_\tau$ commutes with 
$\cos\left(t(\widetilde\Delta_\tau)^{1/2}\right)$, and 
finite propagation speed, we get
\[
\widetilde\Delta_\tau u(t,x;f^\pi_{\tau,v})
=u(t,x;\widetilde\Delta_\tau f^\pi_{\tau,v}).
\]
By (\ref{spherical4}) it follows that
\[
u(t,x;\widetilde\Delta_\tau f^\pi_{\tau,v})
=-(-\pi(\Omega)+\lambda_\tau) u(t,x;f^\pi_{\tau,v}).
\]
Combined with (\ref{partial}) it follows that for every $x\in S$,
$u(t,x;f^\pi_{\tau,v})$ satisfies the following differential equation in $t$
\[
\left(\frac{d^2}{d t^2}-\pi(\Omega)+\lambda_\tau\right)u(t,x;f^\pi_{\tau,v})=0,
\quad u(0,x;f^\pi_{\tau,v})=
f^\pi_{\tau,v}(x),\; u_t(0,x;\phi_\lambda)=0.
\]
This implies the claimed equality.
\end{proof}

Let $\varphi\in\cS(\R)$ be even with $\widehat\varphi\in C^\infty_c(\R)$. Since
the kernel of the integral operator $\varphi((\widetilde\Delta_\tau)^{1/2})$
is given by $h_\varphi\in C^\infty_c(G)$, 
$\varphi((\widetilde\Delta_\tau)^{1/2})(f^\pi_{\tau,v})$ is well defined and it 
follows from Lemma \ref{spherical2} that
\[
\begin{split}
\varphi((\widetilde\Delta_\tau)^{1/2})(f^\pi_{\tau,v})=\frac{1}{\sqrt{2\pi}}
\int_\R\widehat\varphi(t)u(t;f^\pi_{\tau,v})\;dt &= \frac{1}{\sqrt{2\pi}}
\int_\R\widehat\varphi(t)\cos(t\sqrt{-\pi(\Omega)+\lambda_\tau})
f^\pi_{\tau,v}\;dt\\
&=\varphi\left(\sqrt{-\pi(\Omega)+\lambda_\tau}\right) f^\pi_{\tau,v}.
\end{split}
\]
If we rewrite this equality in terms of the kernel $h_\varphi$ and use the 
definition of $f^\pi_{\tau,v}$, we get
\begin{equation}\label{spherical5}
\int_G h_\varphi(g^{-1}g_1)\Phi^\pi_\tau(g^{-1}_1)v\;dg_1=
\varphi\left(\sqrt{-\pi(\Omega)+\lambda_\tau}\right)\Phi^\pi_\tau(g^{-1})v.
\end{equation}
Let $d_\tau:=\dim V_\tau$. Putting $g=1$ and taking the trace of both sides,
we get
\[
\int_G \Tr[h_\varphi(g)\Phi^\pi_\tau(g^{-1})]\;dg=d_\tau
\varphi\left(\sqrt{-\pi(\Omega)+\lambda_\tau}\right).
\]
We continue by rewriting the left hand side. To this end put
\[
\phi^\pi_\tau(g):=\tr\Phi^\pi_\tau(g),\quad g\in G.
\]
Note that $\phi^\pi_\tau$ satisfies $\phi^\pi_\tau(g)=\phi^\pi_\tau(g^{-1})$.
Using the Schur orthogonality relations (see \cite[Chapt. I, \S 5]{Kn}), we get
\begin{equation}\label{spherical6}
\begin{split}
\Phi^\pi_\tau(g)=d_\tau \int_K\Tr[\tau(k^{-1})\Phi^\pi_\tau(g)]\tau(k)\;dk
=d_\tau\int_K \phi^\pi_\tau(k^{-1}g)\tau(k)\;dk.
\end{split}
\end{equation}
Using (\ref{spherical6}), we get
\begin{equation}\label{spherical7}
\begin{split}
\int_G \Tr[h_\varphi(g)\Phi^\pi_\tau(g^{-1})]\;dg&=d_\tau\int_G\int_K
\phi^{\pi}_\tau(k^{-1}g^{-1})\Tr[h_\varphi(g)\tau(k)]\;dk dg\\
&=d_\tau\int_K\int_G \phi^{\pi}_\tau(gk)\tr h_\varphi(gk)\;dg dk\\
&=d_\tau\int_G\tr h_\varphi(g)\phi^\pi_\tau(g)\;dg.
\end{split}
\end{equation}
Together with (\ref{spherical5}) we obtain
\begin{equation}\label{spherical8}
\int_G\tr h_\varphi(g)\phi^\pi_\tau(g)\,dg=
\varphi\left(\sqrt{-\pi(\Omega)+\lambda_\tau}\right).
\end{equation}
Now let $\tau^\prime\in\widehat K$ be any other representation which occurs in
$\pi|_K$. Repeating the argument used in (\ref{spherical7}), we get
\[
\int_G\tr h_\varphi(g)\phi^\pi_{\tau^\prime}(g)\,dg=
\int_G\Tr\left[\left(\int_K\phi^\pi_{\tau^\prime}(k^{-1}g^{-1})\tau(k)\;dk\right) 
h_\varphi(g)\right]\;dg.
\]
Again by the Schur orthogonality relations, we have
\[
\int_K\phi^\pi_{\tau^\prime}(k^{-1}g^{-1})\tau(k)\;dk=0,
\]
if $\tau^\prime\not\cong\tau$. Hence we get
\begin{equation}\label{spherical9}
\int_G\tr h_\varphi(g)\phi^\pi_{\tau^\prime}(g)\:dg=0, \quad \tau^\prime\in
\widehat K,\; \tau^\prime\not\cong\tau.
\end{equation}
Choose an orthonormal basis of $\H_\pi$ which is adapted to the decomposition
of $\pi|_K$ into irreducible representations of $K$. Then it follows from
(\ref{spherical9}) that
\begin{equation}
\begin{split}
\Theta_\pi(\tr h_\varphi)&=\Tr\left[\int_G \tr h_\varphi(g)\pi(g)\;dg\right]
=\sum_{\tau^\prime}\int_G\tr h_\varphi(g)
\phi^\pi_{\tau^\prime}(g)\;dg\\
&=\int_G\tr h_\varphi(g)\phi^\pi_\tau(g)\;dg.
\end{split}
\end{equation}
Combined with (\ref{spherical8}) we obtain the following lemma.
\begin{prop}\label{charact3}
Let $\varphi\in\cS(\R)$ be even with $\hat\varphi\in C^\infty_c(\R)$. Let 
$h_\varphi$ be the kernel of 
$\varphi((\widetilde\Delta_\tau)^{1/2})$. Then
for all $\pi\in\widehat G(\tau)$ we have
\[
\Theta_\pi(\tr h_\varphi)=\varphi\left(\sqrt{-\pi(\Omega)+\lambda_\tau}\right).
\]
\end{prop}
Since $G$ has split rank one, the tempered dual of $G$ (which is the support of
the Plancherel measure) is the union of the unitarily induced representations 
$\pi_{\sigma,\lambda}$, $\sigma\in\widehat M$, $\lambda\in\R$, 
 and the discrete series, where the latter exists only if $\rank G=\rank K$. 
First consider the induced representation $\pi_{\sigma,\lambda}$. Let 
$T\subset M$
be a maximal torus and $\tf$ the Lie algebra of $T$. Let $\Lambda_\sigma
\in i\tf$ be the infinitesimal character of $\sigma\in\widehat M$ and $\rho_M$
the half-sum of positive roots of $(M,T)$. Then by \cite[Proposition 8.22]{Kn}
\begin{equation}\label{casimir2}
\pi_{\sigma,\lambda}(\Omega)=-\lambda^2-|\rho|^2+|\Lambda_\sigma+\rho_M|^2
-|\rho_M|^2,
\end{equation}
where $|\rho|$ is given by (\ref{casimir1}).
Let $\tau\in \widehat K$. By Frobenius reciprocity \cite[p.208]{Kn} we have
\begin{equation}\label{multipl}
[\pi_{\sigma,\lambda}|_K:\tau]=[\tau|_M:\sigma],\quad \sigma\in\widehat M.
\end{equation}
Since we are assuming that $K$ is multiplicity free in $G$, it follows that
$[\tau|_K:\sigma]\le 1$. Let
\[
\widehat M(\tau)=\{\sigma\in\widehat M\colon [\tau|_M:\sigma]=1\}.
\]
Then by (\ref{multipl}) it follows that $\pi_{\sigma,\lambda}\in 
\widehat G(\tau)$ if and only if $\sigma\in\widehat M(\tau)$, and by
 Proposition \ref{charact3} we get
\begin{equation}\label{charact4}
\Theta_{\sigma,\lambda}(\tr h_\varphi)=\varphi\left(\sqrt{\lambda^2+|\rho|^2+
|\rho_M|^2-
|\Lambda_\sigma+\rho_M|^2+\lambda_\tau}\right),\quad \sigma\in \hat M(\tau),\;
\lambda\in\R.
\end{equation}
 Now suppose that $\rank G=\rank K$. Then $G$ has a non-empty discrete series.
Let $H\subset G$ be a compact Cartan subgroup with Lie algebra $\hf$. Let
$L\subset i\hf$ be the lattice of all $\mu\in i\hf$ such that $\xi_\mu(\exp Y)=
e^{\mu(Y)}$, $Y\in \hf_\C$ exists. Let $L^\prime\subset L$ be the subset of 
regular elements. According
to Harish-Chandra the discrete series of $G$ is parametrized by $L^\prime$, 
i.e., for each $\mu\in L^\prime$ there is a discrete series representation
$\pi_\mu$. Moreover $\pi_\mu\cong\pi_{\mu^\prime}$ iff there exists $w\in W$
such that $\mu=w\mu^\prime$, and each discrete series
representation is of the form $\pi_\mu$ for some $\mu\in L^\prime$. Then by
\cite[(6.8)]{Ar} we have
\begin{equation}
\pi_\mu(\Omega)=|\mu+\rho|^2-|\rho|^2,\quad \mu\in L^\prime.
\end{equation}
So Proposition \ref{charact3} gives in this case
\begin{equation}\label{charact5}
\Theta_{\pi_\mu}(\tr h_\varphi)=\varphi\left(\sqrt{|\mu+\rho|^2-|\rho|^2
+\lambda_\tau}\right),\quad \mu\in L^\prime,\;\pi_\mu\in\widehat G(\tau).
\end{equation}
Using the Plancherel formula, (\ref{charact4}) and (\ref{charact5}), we get 
an explicit form of the trace formula (\ref{traceform2}).

Now we consider the case $\tau=1$, where $1$ denotes the trivial 
representation. Then the $h_\varphi$ belongs to the space 
$C^\infty_c(G/\hskip-2pt/K)$ of $K$-bi-invariant, smooth, compactly supported 
functions on $G$. Let $c(\lambda)$ be Harish-Cahndra's c-function. Then the
Plancherel measure for the spherical Fourier transform is given by 
$|c(\lambda)|^{-2}d\lambda$, and the Plancherel formula for spherical functions
(see \cite{He}) gives
\begin{equation}\label{spherical10}
h_\varphi(e)=\frac{1}{2}\int_\R\varphi\left(\sqrt{\lambda^2+|\rho|^2}\right)
|c(\lambda)|^{-2}\,d\lambda.
\end{equation}
Furthermore, note that by Frobenius reciprocity $\widehat M(1)$ consists only
of the trivial representation $1$ of $M$, and by (\ref{charact4}) we have
\begin{equation}\label{charact6}
\Theta_{1,\lambda}(\tr h_\varphi)=\varphi\left(\sqrt{\lambda^2+|\rho|^2}\right)
,\quad \lambda\in\R.
\end{equation}
Inserting (\ref{spherical5})  and (\ref{charact6}) into 
(\ref{traceform2}) we
get the final trace formula for the Laplacian on functions. If we replace 
$\Delta^\#_\chi$ by $\Delta^\#_\chi-|\rho|^2$, then Theorem \ref{scalar} follows.

\section{Restrictions of representations of G}\label{restrict}
\setcounter{equation}{0}

In this section we consider representations of $\Gamma$ which are 
the restriction of a finite-dimensional complex representation 
$\rep\colon G\to \GL(E)$ of $G$. For such representations there exists another
approach to the Selberg trace formula. 

Denote the flat bundle associated to $\rep|_\Gamma$ by $E_\rep$. There is a 
different description of $E_\rep$ as follows. 
Let $E_\tau=\Gamma\bs \widetilde E_\tau$ be the locally homogeneous vector 
bundle associated to the restriction $\tau$ of $\rep$ to $K$. 
Then there is a canonical isomorphism
\begin{equation}\label{iso1}
E_\rep\cong E_\tau
\end{equation}
\cite[Proposition 3.1]{MM}. Note that the space of $C^\infty$-sections of 
$E_\tau$ can be identified with the space $\left(C^\infty(\Gamma\bs G)\otimes 
E\right)^K$ of $K$-invariant vectors in $\left(C^\infty(\Gamma\bs G)\otimes 
E\right)$, where $K$ acts by $k\mapsto R(k)\otimes \rep(k)$, $k\in K$.
Thus there is a canonical isomorphism
\begin{equation}\label{iso2}
\phi\colon C^\infty(X,E_\rep)\cong\left(C^\infty(\Gamma\bs G)\otimes E\right)^K.
\end{equation}

Let $\gf=\kf\oplus\pg$ be the Cartan decomposition.
By  \cite[Lemma 3.1]{MM} there exists a Hermitian inner product $\langle\cdot,
\cdot\rangle_E$ in $E$ which satisfies the following properties.
\[
\begin{split}
&\langle \rep(Y)u,v\rangle_E=-\langle u,\rep(Y)v\rangle_E,\quad 
\mathrm{for}\; Y\in
\kf,\;u,v\in E;\\
&\langle \rep(Y)u,v\rangle_E=\langle u,\rep(Y)v\rangle_E,\quad 
\mathrm{for}\; Y\in
\pg,\;u,v\in E.
\end{split}
\]
In particular, $\langle\cdot,\cdot\rangle_E$ is $K$-invariant. Therefore,
it defines a $G$-invariant Hermitian fiber metric in $\widetilde E_\tau$
which descends to a fiber metric in $E_\tau$. By  (\ref{iso1}) it corresponds
to a fiber metric in $E_\rep$. Let $\Delta_\rep=(\nabla^\rep)^*\nabla^\rep$ be
the associated Laplacian in $C^\infty(X,E_\rep)$. It is a formally self-adjoint
operator. Its spectral decomposition can be determined as follows. 
By Kuga's lemma \cite[(6.9)]{MM} we have
\begin{equation}\label{laplace1}
\Delta_\rep=-R(\Omega)+\rep(\Omega)\Id.
\end{equation}
Assume that $\rep$ is absolutely irreducible. Then there is a scalar 
$\lambda_\rep\ge 0$ such that
\[
\rep(\Omega)=\lambda_\rep\Id.
\]
Let $R_\Gamma$ be the right regular representation of $G$ in $L^2(\Gamma\bs G)$.
Let 
\begin{equation}\label{regrep1}
L^2(\Gamma\bs G)=\widehat\bigoplus_{\pi,\in\widehat G}m_\Gamma(\pi)\H_\pi
\end{equation}
be the decomposition of $R_\Gamma$ into irreducible subrepresentations, where
$\H_\pi$ denotes the Hilbert space of the representation $\pi$. Denote
by $(\H_\pi\otimes E)^K$ the space of $K$ invariant vectors of 
$\H_\pi\otimes E$, where the action of $K$ is given by 
$k\mapsto \pi(k)\otimes \rep(k)$. By (\ref{iso1}) and (\ref{regrep1}) we get
\begin{equation}\label{regrep2}
L^2(X,E_\rep)\cong (L^2(\Gamma\bs G)\otimes E)^K\cong \widehat\bigoplus_{\pi\in
\widehat G}m_\Gamma(\pi)(\H_\pi\otimes E)^K.
\end{equation}
For $\pi\in\widehat G$ let 
\[
\lambda_\pi=\pi(\Omega)
\]
be the Casimir eigenvalue of $\pi$. Then $R(\Omega)$ acts in 
$(\H_\pi\otimes E)^K$ by $\lambda_\pi$. By (\ref{laplace1}) it follows that
w.r.t. the isomorphism (\ref{regrep2}), $\Delta_\rep$ acts in 
$(\H_\pi\otimes E)^K$ as $(-\lambda_\pi+\lambda_\rep)\Id$. Thus (\ref{regrep2})
is the eigenspace decomposition of $\Delta_\rep$. 

Let $\varphi\in \cS(\R)$ be even. By Lemma \ref{smoothing} 
$\varphi\left((\Delta_\rep)^{1/2}\right)$ is a smoothing operator. So it is
a trace class operator. It acts in $(\H_\pi\otimes E)^K$ by
$\varphi\left((-\lambda_\pi+\lambda_\rep)^{1/2}\right)$. Then it follows 
from (\ref{regrep2}) that
\begin{equation}\label{traceform8}
\Tr \varphi\left((\Delta_{\rep})^{1/2}\right)=\sum_{\pi\in\widehat G}
m_\Gamma(\pi)\dim(\H_\pi\otimes E)^K
\varphi\left((-\lambda_\pi+\lambda_\rep)^{1/2}\right).
\end{equation}
To derive the trace formula, we can proceed as in section \ref{locsym}. 
The lift $\widetilde\Delta_\rep$ of $\Delta_\rep$  to $S$ is a $G$-invariant
elliptic differential operator which is symmetric and non-negative. Let
$h_{\rep,\varphi}\colon \Gamma\bs G\to \End(E)$ be the kernel of 
$\varphi\bigl((\widetilde\Delta_\rep)^{1/2}\bigr)$. Applying Proposition 
\ref{traceform} with $\chi=1$  and (\ref{traceform8}), we get
\begin{equation}\label{traceform9}
\begin{split}
\sum_{\pi\in\widehat G}
m_\Gamma(\pi)\dim(\H_\pi\otimes E)^K&
\varphi\left((-\lambda_\pi+\lambda_\rep)^{1/2}\right)=
\vol(\Gamma\bs S)\tr h_{\rep,\varphi}(e)\\
&+\sum_{\{\gamma\}_\Gamma\not=e}\vol(\Gamma_\gamma\bs G_\gamma)
\int_{G_\gamma\bs G}\tr h_{\rep,\varphi}(g^{-1}\gamma g)\;d\dot g.
\end{split}
\end{equation}
{\bf Remark.} Let $\chi=\rep|_\Gamma$. Then we also have the trace formula
of Proposition \ref{traceform} with $\tau=1$. The two formulas are, of course,
different, since the operators are different. In the present case, the 
advantage
is that we can work with self-adjoint operators. On the other hand, the
formula (\ref{traceform1}) is more suitable for applications to Ruelle- and
Selberg zeta functions. 
\hfill$\square$

If the split rank of $G$ is 1, we can use (\ref{orbint1}) to express the 
orbital integrals in terms of characters. This gives
\begin{prop}
Assume that the split rank of $G$ is 1. Let $\rep\colon G\to \GL(E)$ be
an absolutely irreducible finite-dimensional complex representation of $G$.
Let $\varphi\in\cS(\R)$ be even with $\hat\varphi\in C^\infty_c(\R)$. Then
with the same notation as above we have
\begin{equation}\label{traceform10}
\begin{split}
\sum_{\pi\in\widehat G}
m_\Gamma(\pi)\dim(\H_\pi\otimes E)^K&
\varphi\left((-\lambda_\pi+\lambda_\rep)^{1/2}\right)=
\vol(\Gamma\bs S)\tr h_{\rep,\varphi}(e)\\
&+\sum_{\{\gamma\}_\Gamma\not=e}\frac{1}{2\pi}
\frac{l(\gamma_0)}{D(\gamma)}
\sum_{\sigma\in\widehat M}\overline{\tr\sigma(\gamma)}
\int_\R\Theta_{\sigma,\lambda}(h_{\rep,\varphi})\cdot e^{-il(\gamma)\lambda}\;
d\lambda.
\end{split}
\end{equation}
\end{prop}
The characters $\Theta_{\sigma,\lambda}(h_{\rep,\varphi})$ can be computed by the
method explained in section \ref{rank1}.

So there are two classes of finite-dimensional representations of $\Gamma$ 
for which we can work with self-adjoint operators and apply
 the usual Selberg trace formula. These are unitary representations and 
restrictions of rational representations of $G$. In general, not every 
representation of $\Gamma$ belongs to one of these classes. However, if 
$\rk(G)\ge 2$, the superrigidity theorem of Margulis 
\cite[Chapt. VII, \S 5]{Ma} implies that a general representation of $\Gamma$
is not to far from a representation which is either unitary or the restriction
of a rational representation. See  \cite[p. 245]{BW} for more details.

\end{document}